\documentclass[12pt,letterpaper]{article}
\usepackage{epsfig}
\usepackage[]{natbib}
\usepackage{amsmath,amsthm}
\usepackage{amsmath}
\usepackage{amssymb,amsfonts}
\usepackage{nccfoots,bbm,bm}
\usepackage{graphicx,color,caption}
\usepackage{booktabs,multirow,enumerate}
\usepackage{tikz,subfigure}
\usepackage{microtype}
\usepackage[colorlinks,citecolor=blue]{hyperref}
\usepackage[active]{srcltx}
\usepackage{xr}
\usepackage{comment}
\usepackage{longtable}
\usepackage{verbatim}

\usepackage[hmargin=1in,vmargin=1in]{geometry}
\newtheorem{lem}{Lemma}[section]
\newtheorem{defn}{Definition}[section]

\newtheorem{cor}{Corollary}[section]

\newtheorem{thm}{Theorem}[section]
\newtheorem{example}{Example}[section]

\newtheorem{rem}{Remark}[section]
\numberwithin{equation}{section}
\numberwithin{figure}{section}
\numberwithin{table}{section}
\theoremstyle{plain}

\renewcommand{\d}{{\rm{d}}}%% for use in \int \d x   (to make the d upright)

\newcommand{\beq}{\begin{equation}}
\newcommand{\eeq}{\end{equation}}

\newcommand{\var}{\text{\normalfont Var}}
\newcommand{\cov}{\text{\normalfont Cov}}
\newcommand{\ee}{\mathbb{E}}
\newcommand{\pp}{\mathbb{P}}
\newcommand{\rr}{\mathbb{R}}
\newcommand{\nn}{\mathbb{N}}
\newcommand{\zz}{\mathbb{Z}} 
\newcommand{\qq}{\mathbb{Q}}

\def\qed{\rule{2mm}{2mm}}

\usepackage[section]{placeins}

\usepackage{graphicx}
\graphicspath{ {./Plots}
}

\begin{document}

\title{Permutation Testing for Monotone Trend}

%\author{Joseph P. Romano\footnote{Supported by NSF Grant MMS-1949845.}
 % \\
%Departments of Statistics and Economics \\
%Stanford University \\
%\href{mailto:romano@stanford.edu}{romano@stanford.edu}
%\and
%Marius A. Tirlea\footnote{Supported by NSF Grant MMS-1949845.} \\
%Department of Statistics \\
%Stanford University \\
%\href{mailto:mtirlea@stanford.edu}{mtirlea@stanford.edu} 
%}

\author{Joseph P. Romano\footnote{Supported by NSF Grant MMS-1949845.}
 \\
Departments of Statistics and Economics \\
Stanford University \\
\href{mailto:romano@stanford.edu}{romano@stanford.edu}
\and
Marius A. Tirlea\footnote{Supported by NSF Grant MMS-1949845.} \\
Department of Statistics \\
Stanford University \\
\href{mailto:mtirlea@stanford.edu}{mtirlea@stanford.edu} 
}

\date{\today}

\maketitle

\bigskip

\begin{abstract}
In this paper,  we consider the fundamental problem of testing for monotone trend in a time series.
While the term ``trend" is commonly used and has an intuitive meaning, it is first crucial to specify its exact meaning in a hypothesis testing context.  A commonly used well-known test is the Mann-Kendall test, which  we show does not offer Type 1 error control even in large samples.  On the other hand, by
an appropriate studentization of the Mann-Kendall statistic,  we construct permutation tests 
that offer asymptotic error control quite generally, but retain the exactness property of permutation tests for i.i.d. observations.  We also  introduce  ``local" Mann-Kendall statistics as a means of testing for local rather than global trend in a time series.  Similar properties of permutation tests are obtained for these tests as well.
\end{abstract}

\bigskip

\bigskip

\section{Introduction}

In both theoretical and applied time series analysis, the issue of determining whether or not a given sequence of observations exhibits monotone trend is a crucial element of understanding the process under study.   An important  nonparametric method used for testing monotone trend is the Mann-Kendall trend test (see \cite{mann1945} and \cite{kendall}), which tests the hypothesis 

$$
H_{\text{MK}}: X_1, \, \dots,\, X_n \text{ i.i.d. } 
$$

\noindent using a rank statistic (defined in (\ref{equation:mann-kendall})). Under the null hypothesis $H_{\text{MK}}$, the Mann-Kendall test  of nominal level $\alpha$  is exact in finite samples and asymptotically valid. However, while it is the case that an i.i.d. sequence of random variables $\{X_i : i \in [n]\}$ \textit{should} intuitively be described as exhibiting no monotone trend, it is also intuitively reasonable that a non-i.i.d. sequence $\{X_i: i \in [n]\}$ may also exhibit a lack of monotone trend. Indeed, in the time series contexts in which such problems are usually considered, there is an implicit dependence structure in time of the sequence under examination, in which case the null hypothesis $H_{\text{MK}}$ is \textit{not} the null hypothesis intended to be under consideration. On account of this, significant problems of error control may arise, since the null hypotheses of i.i.d. and ``no monotone trend" are quite different. This may lead to issues of Type 1 and Type 3 (or directional) error control analogous  to those described in \cite{tirlea} for testing autocorrelations.  In the context of this paper,  the issue is that one can reject the null hypothesis and conclude that there exists a monotone trend not because there does truly exist some monotone trend, but on account of the test only controlling Type 1 error for i.i.d. sequences, and not for weakly dependent sequences exhibiting no monotone trend.

We propose two nonparametric testing procedures for the hypothesis 

$$
H: \{X_i: i \in \nn\} \text { exhibits no monotone trend,} 
$$

\noindent where several precise definitions of the null hypothesis $H$ are considered. In particular, we define the notions of \textit{global} and \textit{local} monotone trend, and consider permutation testing procedures based on the full, or global, Mann-Kendall statistic. We also introduce a novel statistic, termed the local Mann-Kendall statistic, which is used for testing the null hypothesis of a lack of local monotone trend.

In the broader context of testing for and estimation of monotone trend, \cite{dietz_killeen} provide a multivariate test of trend with application to a pharmaceutical setting, \cite{zhao_woodroofe} consider isotonic estimators of sequences with monotone trend under stationarity, and \cite{hanasymptotics} extend the Mann-Kendall test to the setting of independent, but not necessarily i.i.d. sequences. In a more applied setting, \cite{yue2002} discuss application of the Mann-Kendall and Spearman's $\rho$ test to hydrological time series, \cite{hamed_2008} considers a modification of the Mann-Kendall test under scaling, and \cite{power_mk_test} consider the power of different versions of the Mann-Kendall test 
in a simulation study.
Permutation tests of trend based on OLS regression are discussed in \cite{tirlea-OLS}.

There has been a resurgence in the use of permutation and randomization tests,
as they offer the prospect of valid inference in complex settings.  However, they can often fail (even in large samples) without careful implementation.  Indeed, this paper is part of a growing body of work that shows how permutation tests can indeed be constructed that offer exact finite-sample validity under a restricted null hypothesis, but also offer asymptotic validity in much more generality.   Many such instances of this phenomenon are given in \cite{tirlea}.

Section \ref{sec.def.monotone} provides a precise definition of the null hypotheses of a lack of monotone trend under consideration. The main results for the global Mann-Kendall test are given in Section \ref{sec.global}, in which we provide conditions for the asymptotic validity of the permutation test when $\{X_n: n \in \nn\}$ is a stationary, absolutely regular sequence, satisfying fairly standard mixing conditions. For instance, the results in this section may be applied freely to a large class of stationary ARMA sequences. Section \ref{sec.local} introduces the local Mann-Kendall statistic, and furnishes the main results for asymptotic validity of the permutation test when $\{X_n: n \in \nn\}$ is a stationary, $\alpha$-mixing sequence. Section \ref{sec.simulations.monotone.trend} provides simulation studies illustrating the results of the previous sections. The majority of the proofs are deferred to the supplement, owing to the length and technical requirements required to prove the results.

Most of our results require some notions of stationarity, weak dependence, and absolute regularity, whose definitions we now review.
A sequence $\{X_n: n \in \nn\}$ is called stationary if it is strongly stationary, i.e. if, for all $k \in \nn$, for all $\{n_i: i \in [k]\} \subset \nn$, and for all $m \in \nn$,  

$$
(X_{n_1}, \, \dots, \, X_{n_k}) \overset{d}= (X_{m + n_1}, \, \dots, \, X_{m + n_k}) \, \, .
$$

\noindent With these notational conventions, we turn to a brief discussion of the notion of weak dependence in a sequence of random variables.

\begin{defn} \rm Let $\{X_n:n\in \nn\}$ be a sequence of $(\Omega, \, \mathcal{F}, \, \pp)$-measurable random variables. For each $n \in \nn$, let $\mathcal{F}_n = \sigma(X_m: m \leq n)$, and let $\mathcal{G}_n = \sigma(X_m: m \geq n)$. For $n \in \nn$, let $\alpha_X(n)$ be Rosenblatt's $\alpha$-mixing coefficient, 

$$
\alpha_X(n) = \sup_{m \in \nn} \sup_{A \in \mathcal{F}_m, \,  B \in \mathcal{G}_{m+ n} }\left| \pp(A \cap B) - \pp(A) \pp(B) \right| \, \, .
$$

\noindent We say that the sequence $\{X_n\}$ is strongly mixing or $\alpha$-mixing if $\alpha_X(n)\to 0$ as $n\to\infty$.

For $n \in \nn$, let $\beta_X(n)$ be the $\beta$-mixing or absolute regularity coefficient, defined as 

$$
\beta_X (n) = \sup_{m \in \nn} \ee \left[ \sup_{B \in \mathcal{G}_{m + n} } \left| \pp(B \, | \, \mathcal{F}_m ) - \pp(B) \right| \right] \, \, .
$$ 

\noindent We say that the sequence $\{X_n\}$ is absolutely regular or $\beta$-mixing if $\beta_X( n)\to0$ as $n\to\infty$. 

\end{defn}

\begin{rem} \rm As discussed in \cite{bradley2005}, we have that, for any sequence $\{X_n: n \in \nn\}$ of $\rr^d$-valued random variables, and for each $n \in \nn$,

$$
2\alpha_X(n) \leq  \beta_X (n) \, \, .
$$

\noindent It follows that any $\beta$-mixing sequence is also $\alpha$-mixing. \qed 
\end{rem}

\section{Defining monotone trend}\label{sec.def.monotone} 

We turn our attention to establishing the meaning of the expression ``monotone trend" used in  the remainder of this paper. In \cite{mann1945}, it is the case that monotone trend is defined solely in contrast to the null of i.i.d., and is not explicitly defined in itself. However, appropriately defining this phrase, or, more pertinently, defining a lack of monotone trend, is essential in order to conduct a hypothesis test of a lack of monotone trend, since otherwise defining the null becomes impossible. 

In this work, we restrict attention to consideration of weakly dependent sequences,  as defined earlier.  When considering the monotonicity (or lack thereof) of the sequence $\{X_n: n \in \nn\}$, we provide some examples of the difficulties arising when using traditional notions of monotonicity, such as $\ee[X_i] < \ee[X_j]$ for $i < j$.

\begin{example}\label{ex.unbounded} (Positive local trend, negative global trend) \rm Let $c, \, \epsilon > 0$. Consider the sequence $\{Y_n: n \in \nn\}$ such that $Y_0 = 0$, and, for all $n \geq 1$, 

$$
Y_n =\sum_{i=1}^n X_i \, \, ,
$$

\noindent where the $X_i$ are i.i.d. such that 

$$
X_i = \begin{cases} 
1, \, \, \text{with probability } 1 - \epsilon \, \, ,\\
-c, \, \, \text{with probability } \epsilon \, \, .
\end{cases}
$$

\noindent Let $\delta > 0$, and let $M \in \nn$. By letting $\epsilon = (1- \delta)^{1/M}$ and $c$ be sufficiently large, we have that $\{Y_n\}$ is a (strictly) decreasing sequence in expectation, and so this sequence  exhibits a negative \textit{global} monotone trend with probability 1, since, by the strong law of large numbers, 

$$
\frac{1}{n} Y_n \overset{a.s.}\to \ee[X_1] < 0 \, \, , 
$$

\noindent from which it follows that, for $i < j$, as $j-i \to \infty$, $\pp(Y_i > Y_j ) \to 1$, and in fact, for any constant $K \in \rr^+$, $\pp(Y_i > Y_j + K) \to 1$ as $j - i \to \infty$.

\begin{figure}[h!]
\begin{center}
\begin{tabular}{c}
{\mbox{\epsfxsize=100mm\epsfbox{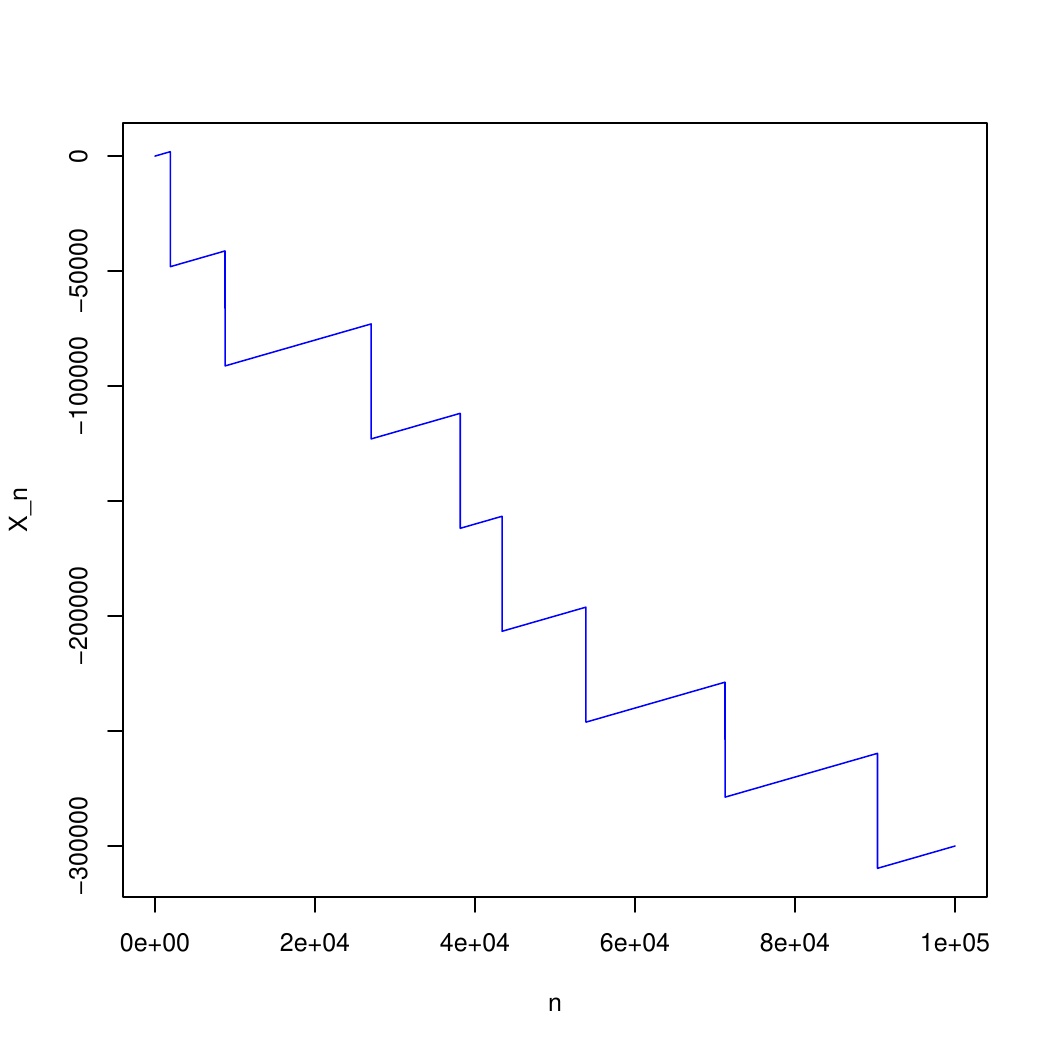}}} 
\end{tabular}
\end{center}
\caption{A sample path from the process described in Example \ref{ex.unbounded}, with $n = 10^5$, $\epsilon = 10^{-4}$, and $M = 5 \cdot 10^4$.} 
\label{fig.unbounded.ex}
\end{figure}

\noindent However, such a sequence exhibits a \textit{local} monotone trend in the opposite direction in the following sense: for any $n \in \nn$, we have that 

$$
\pp( X_n < X_{n+1} < \dots <X_{n + M- 1}) = 1- \delta \, \, .
$$

\noindent To conclude, since it is possible to construct such a sequence for any choice of $\delta$ and $M$, we may construct a sequence $\{X_n: n \in \nn\}$ with the following property: for any $i \in \nn$, the sequence $(X_i, \, \dots, \, X_{i+M- 1})$ is strictly increasing with probability $(1- \delta)$, but $X_n \to -\infty$ with probability 1. \qed

\end{example}

Example \ref{ex.unbounded} indicates that some distinction is required between the notions of local and global monotone trend. In particular, as will be illustrated in Section \ref{sec.global}, strictly stationary and weakly dependent (in particular, absolutely regular) sequences cannot, by definition, exhibit global trend, in the sense that, as $|j - i| \to \infty$,

$$
\pp(X_i \leq X_j) - \pp(X_i \geq X_j) \to 0 \, \, .
$$

\noindent On account of this property of strictly stationary and weakly dependent sequences, it follows that any test of monotone trend in this context is, in fact, a test of stationarity, since one cannot disentangle the properties of stationarity and lack of monotone trend (see Remark \ref{stat.no.global.trend}). While there is an argument to be made that the restriction of stationarity is too strong in such a setting, this assumption is necessary to ensure some limiting behaviour of functions of the sequence $\{X_n: n \in \nn\}$; indeed, even without considering weakly dependent processes, there exist sequences of independent, uniformly bounded random variables for which, for example, the sample mean does not converge in distribution.

In light of this, we therefore consider a test of global monotone trend to test the null hypothesis 

$$
H_0: \{X_n: n \in \nn\} \text{ is stationary,}
$$

\noindent where, for methods developed later in this paper,  we consider the power of such tests against alternatives the form of which is explicitly specified.

Fortunately, the issue of appropriately defining the notion of local monotone trend is a simpler one; in particular, local monotonicity is a property which is not directly tied to the stationarity of a sequence. For a simple example of a sequence with no local monotone trend, consider $X_n$ being i.i.d. random variables drawn from a continuous distribution $F$. In this setting, which is the null setting of the original test of \cite{mann1945}, any ordering of any local subsequence of this process $(X_n, \, \dots,\, X_{n+M-1})$ is equally likely, and any definition of local monotone trend should exclude this example from exhibiting such a trend. In contrast, a sequence for which it should be said that a local monotone trend exists is given in the example below.

\begin{example}\label{ex.markov.chain}(Markov chain with local monotone trend) \rm Let $M \in \nn$, and let $\epsilon > 0$. Let 

$$
\pi_{-M} = \frac{\epsilon}{1 - (1-\epsilon)^{2M+1}} \, \, .
$$

\noindent For each $i \in \{-M + 1, \, \dots, \, M\}$, let 

$$
\pi_i = \pi_{-M} \cdot (1-\epsilon)^{i + M}
$$

\noindent Consider the Markov chain $\{X_n: n \in \nn\}$ on the state space $\{-M, \, \dots,\, M\}$ with initial distribution $\pi$, and transition matrix 

$$
P_{i, \, j} = \begin{cases} 
1- \epsilon, \, \, &j = i+1 \text{ and } i \leq M-1 \\
1- \epsilon, \, \, &i = j = M \\
\epsilon, &j = -M \, \, .
\end{cases}
$$

\noindent This is an irreducible, aperiodic Markov chain on a finite state space with stationary distribution $\pi$. In particular, by the weak Markov property, this sequence is strictly stationary, and is absolutely regular, as a consequence of standard convergence results for irreducible, aperiodic Markov chains on a finite state space (see \cite{freedman}). A sample path from such a process is shown below.

\begin{figure}[h!]
\begin{center}
\begin{tabular}{c}
{\mbox{\epsfxsize=100mm\epsfbox{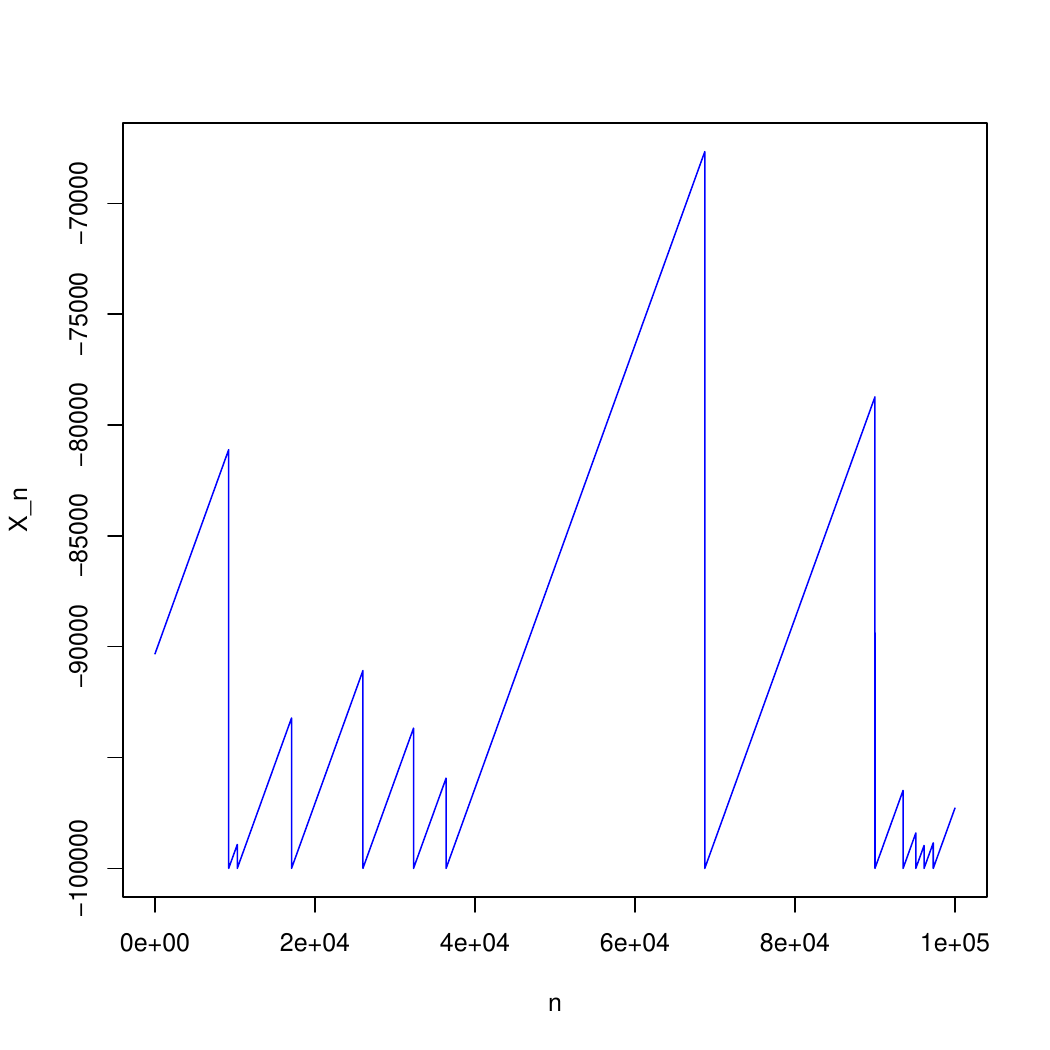}}} 
\end{tabular}
\end{center}
\caption{A sample path from the process described in Example \ref{ex.markov.chain}, with $n = 10^5$, $\epsilon = 10^{-4}$, and $M = 10^5$.} 
\label{fig.bounded.ex}
\end{figure}

For an appropriate choice of $\epsilon$ and $M$, this process exhibits the same local behaviour as the process in Example \ref{ex.unbounded}; namely, for any arbitrary run length $K \in \nn $ and any arbitrarily small $\delta > 0$, one can construct such a process for which, for any $n$,

$$
\pp( X_n < X_{n+1} <\dots < X_{n+K}) \geq (1- \delta) \, \, ,
$$

\noindent i.e. one may construct a sequence with arbitrarily strong local monotone trend. \qed

\end{example}

In light of these examples, we proceed to define the null hypotheses of no global monotone trend and no local monotone trend. In order to test the hypothesis of no global monotone trend in the setting of weakly dependent sequences, we use the null hypothesis 

$$
H_0^{(g)}: \{X_n\} \text{ is strictly stationary.} 
$$

\noindent In order to test the hypothesis of no local monotone trend of order $M$ in the setting of strictly stationary, weakly dependent sequences, we use the null hypothesis 

$$
H^{(l)}_{0, \, M}: \frac{2}{(n-M)M} \sum_{i < j \leq i + M} \left( \pp(X_i < X_j) - \pp(X_i > X_j) \right) = 0 \, \, .
$$

\noindent Of course, under strict stationarity, this is equivalent to the null hypothesis 

$$
\tilde{H}^{(l)}_{0,\, M}: \sum_{i=2}^{M+1} (\pp(X_1 < X_i) - \pp(X_1 > X_i) ) = 0 \, \, .
$$

\noindent With these null hypotheses precisely defined, we may now begin construction of appropriate permutation testing procedures for both global and local monotone trend.

\section{Testing for global trend}\label{sec.global}

In this section, we consider the problem of testing the null hypothesis 

$$
H^{(g)}_0: \{X_i: i \in [n]\} \text{ is stationary,} 
$$

\noindent in the setting where the sequence $\{X_i: i \in [n]\}$ is absolutely regular. When the distribution of $(X_1, \, \dots,\, X_n)$ is invariant under permutation, i.e. the sequence is exchangeable, the randomization hypothesis holds, and so one may construct permutation tests of the hypothesis $H_0^{(g)}$ with exact level $\alpha$. However, in the case of absolutely regular sequences, by Lemma S.3.1 in \cite{tirlea}, exchangeability and i.i.d. are equivalent conditions.  This implies that any permutation testing procedure will retain the property of exactness under the additional assumption that the $X_i$ are i.i.d., and this is the only condition under which the randomization hypothesis holds in this setting. 

However, if the realizations of the sequence $\{X_i: i \in [n]\}$ are not independent or are not i.i.d., the test may not be valid even asymptotically, i.e. the rejection probability of a permutation test may not be equal to $\alpha$ or even close to $\alpha$ as $n\to \infty$. Our goal is therefore to construct a testing procedure which has asymptotic rejection probability equal to $\alpha$, but which also retains finite sample exactness under the additional assumption that the $X_i$ are i.i.d., and so appropriate consideration of the asymptotic properties of the permutation distribution and the test statistic must be undertaken. We wish to consider a permutation test based on the Mann-Kendall statistic, and so we may begin by defining the Mann-Kendall statistic and analyzing the limiting behavior of the permutation distribution $\hat{R}_n$ based on this test statistic.

\subsection{The global Mann-Kendall statistic}\label{sec.full}

\begin{defn} \rm For $X_1, \, \dots,\, X_n$ a sequence of random variables, let 

$$\label{equation:mann-kendall}
U_n = U_n (X_1, \, \dots, \, X_n) = \binom{n}{2} ^{-1} \sum_{i = 1} ^n \sum_{j = i} ^{n} \left( I\{ X_j > X_i \} -  I\{ X_i> X_j \}\right) \, \, .
$$

\noindent be the (global) Mann-Kendall statistic. 

\label{mk.defn} 
\end{defn}

\begin{thm} Let $\{X_i: i \in \nn\}$ be a sequence of random variables such that, for all $i \neq j$, $\pp(X_i =X_j ) = 0$.

\noindent Let $\hat{R}_n$ be the permutation distribution, based on the test statistic $\sqrt{n} U_n$, with associated group of transformations $S_n$, the symmetric group of order $n$. We have that, as $n \to \infty$,

$$
\sup_{t \in \rr} \left| \hat{R}_n (t) - \Phi(3t/2) \right | \, \, ,
$$

\noindent where $\Phi$ is the standard Gaussian c.d.f. 

\label{basic.perm.lim}
\end{thm} 

\begin{proof} Let $\Pi_n \sim \text{Unif}(S_n)$, independent of $\{X_i, \, i \in [n]\}$. Let 

$$
U_{\Pi_n} = U_n( X_{\Pi_n(1)}, \, \dots, \, X_{\Pi_n(n)} ) \, \, .
$$

\noindent Conditional on the sequence $\{X_i, \, i \in [n] \}$, we observe that, on account of the lack of ties, we have that each ordering of the $X_{\Pi_n(i)}$ is equally likely. Since the distribution of $U_n$ only depends on the ranks of the sequence $\{X_i: i \in [n]\}$, the result follows by \cite{mann1945}.
\end{proof}

We observe that, other than the assumption of no ties among the $\{X_i\}$, no other conditions are required for the result of Theorem \ref{basic.perm.lim}. This occurs since $U_n$ is a rank statistic, and, under the action of a random permutation $\Pi_n$, the ranks of the $\{X_i\}$ are uniformly distributed over the set of permutations $S_n$, and so the permutation distribution $\hat{R}_n$ is exactly equal to the distribution of $U_n$ under the when the $X_i$ are independent and identically distributed. 

In order to appropriately assess the asymptotic validity of the permutation test based on this statistic, we must turn our attention to determining the asymptotic distribution of $U_n$.

\subsection{Asymptotic distribution of the Mann-Kendall statistic}

Having determined the  limiting behavior of the permutation distribution based on $U_n$, we now consider the true unconditional limiting  distribution of the test statistic, since consistency would require these
limiting distributions be the same. 
Due to allowing quite general forms of weak dependence, the  asymptotic distribution of the test statistic presents  more of a challenge.
We proceed as follows: although $U_n$ is not a U-statistic, since the corresponding kernel $h$ is antisymmetric, we may apply the same idea behind the projection method used to find the asymptotic distribution of U-statistics as in the original proof due to \cite{hoeffding}, i.e. we linearize the statistic $U_n$. 
Note, however, the linearization is not based on the true ``projection" based on conditional expectations, but 
rather based on a ``pseudo-projection", which assumes the observations are i.i.d. In order to perform this linearization, we require the following definition, as well as the subsequent lemma.

\begin{defn} \rm For $\{i_j: j \in [k]\} \subset \nn$, let $\{X_{i_j}: j \in [k]\}$ be a sequence of random variables. For each $j \in \{1, \, \dots, \, k-1\}$, let $\pp_{j} ^{(k)}$ be the probability measure defined by 

$$
\pp_j ^{(k)} ( E^{(j)} \times E^{(k-j)}) = \pp( (X_{i_1}, \, \dots, \, X_{i_j}) \in E^{(j)})\pp((X_{i_{j+1}}, \, \dots, \, X_{i_k}) \in E^{(k-j)}) \, \, ,
$$

\noindent where $E^{(j)}$ and $E^{(k-j)}$ are Borel sets in $\rr^j$ and $\rr^{k-j}$, respectively. Also, let $\pp_0^k$ be the probability measure defined by 

$$
\pp_0 ^{k} (E^{(k)}) = \pp((X_{i_1}, \, \dots, \, X_{i_k}) \in E^{(k)}) \, \, , 
$$

\noindent where $E^{(k)}$ is a Borel set in $\rr^k$. 

\label{defn.indep}
\end{defn}

\begin{lem} Let $h:\rr^k \to \rr$ be a Borel function such that $|h| \leq M$, for some $M \in \rr^+$. Then 

$$
\left| \int_{\rr^k} h(x_1, \, \dots,\, x_k) \d\pp^{(k)}_0 - \int_{\rr^k} h(x_1, \, \dots, \, x_k) \d\pp_j ^{(k)} \right| \leq 2M \beta_X(i_{j+1} - i_j) \, \,.
$$

\label{lem.bound}
\end{lem}

\begin{proof} Let $\nu$ be the signed measure $\pp_0 ^{(k)} - \pp_j ^{(k)}$, and let $|\nu|$ denote its corresponding total variation measure. We have that 

$$
\begin{aligned}
\left| \int_{\rr^k} h(x_1, \, \dots,\, x_k) \d\pp^{(k)}_0 - \int_{\rr^k} h(x_1, \, \dots, \, x_k) \d\pp_j ^{(k)} \right| &= \left| \int_{\rr^k} h(x_1, \, \dots,\, x_k) \d\nu \right| \\
&\leq \int_{\rr^k} |h(x_1,\,\dots,\, x_k) | \d|\nu| \\
&\leq 2M \cdot \text{TV}\left(\pp_0^{(k)}, \, \pp_j ^{(k)} \right) \\
&= 2M \beta_X(i_{j+1} - i_j) \, \, ,
\end{aligned}
$$

\noindent where TV denotes total variation distance, and the final equality follows by \cite{rozanov}. \end{proof}

We utilize this lemma as follows. Let $F$ be the marginal distribution of the $X_i$. In essence, we may not perform a true linearization, since, for $i \neq j$, the conditional expectation $\pp(X_i \leq X_j \, | \, X_j)$ is not equal to $F(X_j)$, since the $X_i$ need not be independent. In spite of this, Lemma \ref{lem.bound} provides an upper bound on the difference between the true projection 

$$
\pp(X_i \leq X_j \, |\, X_j) 
$$

\noindent and the approximate projection $F(X_j)$, in terms of the $\beta$-mixing coefficients of the sequence $\{X_n: n \in \nn\}$. By requiring appropriate $\beta$-mixing conditions on the sequence $\{X_i\}$, we may bound the error term of this approximate linearization, and, using the central limit theorem due to \cite{neumann}, we may conclude the following result.

\begin{thm} Let $\{X_n : n \in \nn\}$ be a strictly stationary, absolutely regular sequence of real-valued random variables with marginal distribution $F$, such that, for all $i\neq j$, $\pp(X_i = X_j) = 0$. Suppose that the $\beta$-mixing coefficients of $\{X_n\}$ satisfy

\beq
\sum_{n=1}^\infty \beta_X(n) < \infty \, \, .
\label{beta.finite}
\eeq

\noindent For each $i \in [n]$, let $V_i = 1 - 2F(X_i)$. Let 

\beq
\sigma^2 = \frac{4}{9} + \frac{8}{3} \sum_{k \geq 1} \cov(V_1, \, V_{1+ k} ) \, \, .
\label{lim.variance.centre}
\eeq

\noindent Suppose that $\sigma^2 > 0$. As $n \to \infty$, 

$$
\sqrt{n} U_n \overset{d}\to N(0, \, \sigma^2) \, \, .
$$

\label{thm.null.conv}
\end{thm}

\begin{rem} \rm Note that, as a consequence of the proof of Theorem \ref{thm.null.conv}, we have that 

$$
\ee [\sqrt{n} U_n ] = o( 1) \, \, ,
$$

\noindent i.e. for any stationary, absolutely regular sequence satisfying the conditions laid out therein, the limiting mean of the test statistic is 0. In particular, this motivates our choice of null hypothesis $H_0^{(g)}$, since this will be the case for any stationary sequence $\{X_n: n \in \nn\}$ satisfying the requisite mixing conditions. \qed

\label{stat.no.global.trend}
\end{rem}

\begin{rem} \rm Since, for any sequence $\{X_n: n \in \nn\}$, we have that, for all $n$,

$$
\beta_X (n) \geq 2\alpha_X (n) \, \, ,
$$

\noindent it follows that the condition (\ref{beta.finite}) implies that 

$$
\sum_{n \geq 1} \alpha_X(n) < \infty \, \, .
$$

\noindent This condition is sufficient for Theorem 2.1 of \cite{neumann} to be applied in the above proof. It also follows that 

$$
\sum_{k=1} ^\infty \beta_X (k) < \infty \implies \sum_{k=1} ^n k \beta_X(k) = o(n) \, \, ,
$$

\noindent a proof of which is as follows. For $(k, \, l) \in \nn \times (\nn \cup \{\infty\})$ such that $k \leq l$, let 

$$
S_{k, \, l} = \sum_{i = k } ^l \beta_X (i) \, \, .
$$

\noindent Let $\epsilon > 0$. Without loss of generality, we may assume that $\epsilon < S_{1, \, \infty}$. There exists $N \in \nn$ such that, for all $k \geq N$, 

$$
S_{k, \, \infty} < \epsilon \, \, .
$$

\noindent We have that, for $n \geq N$, 

$$
\begin{aligned}
\sum_{k=1} ^n k \beta_X( k) &= \sum_{k=1} ^n S_{k, \, n} \\
&= \sum_{k=1}^{N-1} S_{k, \, n} + \sum_{k = N}^n S_{k, \, N} \\
&\leq (N-1) S_{1, \, \infty} + (n - N+1) \epsilon \\
&\leq 2n\epsilon \, \, ,
\end{aligned}
$$

\noindent for all $n \geq N S_{1, \, \infty}/\epsilon$. 

However, due to the above two implications, it may be seen that the condition (\ref{beta.finite}) may be replaced with the weaker conditions 

\beq
\begin{aligned}
\sum_{k= 1}^\infty \alpha_X(k) &< \infty \\
\sum_{k=1}^n k \beta_X(k) &= o(n) \, \, ,
\end{aligned}
\label{beta.finite.weaker}
\eeq

\noindent without affecting the result of Theorem \ref{thm.null.conv}. \qed

\label{rem.beta.mix}
\end{rem}

\begin{rem} \rm The proof of Theorem \ref{thm.null.conv} follows the same structure as the original proof of the limiting distribution of U-statistics found in \cite{hoeffding}. The statistic in question is split into a linearized term, formed by conditioning on one of the entries in the kernel $h$, and a remainder term. It is then shown that the remainder term converges to 0 in probability, and that the linearized term satisfies a central limit theorem.

There are two immediate issues with this approach in the case of the Mann-Kendall statistic for dependent data: firstly, the kernel $h$ is antisymmetric, hence the order of projection becomes relevant to the resulting linearized term. In the case of independent, but not necessarily i.i.d. random variables, this has been previously considered in \cite{hanasymptotics}, with an application to the Mann-Kendall statistic.

The second, and more challenging, issue is that, in the case of dependent data, the conditional expectation 

\beq
\ee[h(X_i, \, X_j) \, | \, X_j]  
\label{conditional.exp}
\eeq

\noindent depends not only on $X_j$, but also on the conditional distribution of $X_i$ given $X_j$. In particular, in the case of a strictly stationary sequence, (\ref{conditional.exp}) depends on both $X_j$ and $j- i$. This suggests that a true projection, formed by conditioning on the variables $\{X_i\}$, will not lead to a sum of $n$ random variables which is more amenable to the application of a central limit theorem.

However, as the proof of Theorem \ref{thm.null.conv} demonstrates, one may take the approach of pseudo-projection: namely, we may rewrite the full statistic as a sum of linear terms which would be obtained by projection if the data were i.i.d., in addition to a remainder term. Heuristically, the reasoning behind this approach is as follows. By Lemma \ref{lem.bound}, the difference between the true projection and pseudo-projection for any given term in the Mann-Kendall statistic is 

$$
\ee \left[ | (1- 2F(X_j)) - \ee[h(X_i, \, X_j) \, | \, X_j ] | \right] = O( \beta_X (j-i) )\, \, .
$$

\noindent By Markov's inequality, the sum of the differences is therefore ``small in probability" under the given conditions, and so a CLT may be applied to the sum of pseudo-projections without penalty. \qed
\end{rem}

While we have shown that a central limit theorem holds for the statistic $U_n$, we observe, that, in general, the limiting variance (\ref{lim.variance.centre}) of the global Mann-Kendall statistic is not, in general, equal to the limiting variance of the permutation distribution, as found in Theorem \ref{basic.perm.lim}. In particular, this implies that a permutation test based on this statistic does not control the probability of a Type 1 error, even asymptotically. We therefore turn our attention to studentization, i.e. we wish to find an appropriate estimator of the variance (\ref{lim.variance.centre}). To this end, we provide the following theorem.

\begin{thm} Let $\{X_n: n \in \nn\}$ be a sequence of random variables satisfying the conditions in Theorem \ref{thm.null.conv}. Let $\hat{F}_n$ be the empirical distribution of $\{X_n: n \in \nn\}$. Let $\{b_n: n \in \nn\} \subset \nn$ be a sequence such that $b_n < n$ for all $n$, $b_n = o(\sqrt{n})$, and, as $n \to \infty$, $b_n \to \infty$. Let 

\beq
\hat{\sigma}_n ^2 = \frac{4}{9} + \frac{8}{3n} \sum_{k=1}^{b_n} \sum_{j= 1}^{ n-k} \left(1 -2 \hat{F}_n(X_j) \right) \left(1 -  2 \hat{F}_n (X_{j + k})\right) \, \, .
\label{var.est.global}
\eeq

\noindent Let $\sigma^2$ be as in Theorem \ref{thm.null.conv}. Then, as $n\to \infty$, 

$$
\hat{\sigma}_n^2 \overset{p} \to \sigma^2 \, \, .
$$

\label{thm.consistency.var}
\end{thm}

Having constructed a consistent estimator of the variance, we may now studentize  to obtain the following result.

\begin{thm} Let $\{X_n: n \in \nn\}$ be a strictly stationary, absolutely regular sequence of random variables, with marginal distribution $F$, such that, for all $i \neq j$, $\pp(X_i = X_j) = 0$. Suppose that the $\beta$-mixing coefficients of $\{X_n\}$ satisfy 

$$
\sum_{n=1} ^\infty \beta_X (n) < \infty \, \, .
$$

\noindent For each $i \in \nn$, let $V_i = 1- 2F(X_i)$. Let $\sigma^2$ be as in (\ref{lim.variance.centre}), and suppose that $\sigma^2 > 0$. Let $\{b_n: n \in \nn\} \subset \nn$ be a sequence such that $b_n < n$ for all $n$, $b_n = o(\sqrt{n})$, and, as $n \to \infty$, $b_n \to \infty$. Let 

$$
\hat{\sigma}_n ^2 = \frac{4}{9} + \frac{8}{3n} \sum_{k=1}^{b_n} \sum_{j= 1}^{ n-k} \left(1 -2 \hat{F}_n(X_j) \right) \left(1 -  2 \hat{F}_n (X_{j + k})\right) \, \, .
$$

\noindent The following results hold.

\begin{enumerate}

\item[(i)] As $n \to \infty$, 

$$
\frac{\sqrt{n} U_n}{\hat{\sigma}_n} \overset{d}\to N(0, \, 1) \, \, .
$$

\item[(ii)] Let $\hat{R}_n$ be the permutation distribution, with associated group of transformations $S_n$, the symmetric group of order $n$, based on the test statistic $\sqrt{n}U_n/\hat{\sigma}_n$. Then, as $n \to \infty$, 

$$
\sup_{t \in \rr} \left| \hat{R}_n (t) - \Phi (t) \right| \overset{a.s.}\to 0 \, \, ,
$$

\noindent where $\Phi$ is the standard Gaussian c.d.f. 

\end{enumerate}

\label{thm.perm.test}
\end{thm}

\begin{proof} The result of (i) follows immediately from Theorem \ref{thm.null.conv}, Theorem \ref{thm.consistency.var}, and Slutsky's theorem. We turn our attention to (ii). 

Let $\Pi_n, \, \Pi_n' \sim \text{Unif}(S_n)$, with $\Pi_n$, $\Pi_n'$, and $(X_1, \, \dots, \, X_n)$ independent, and let $\mathbf{X}_{\Pi_n}$, $\mathbf{X}_{\Pi_n'}$ denote the action of $\Pi_n$ and $\Pi_n'$ on $(X_1, \, \dots, \, X_n)$, respectively. Since $F$ is continuous, the ranks of the sequences $\mathbf{X}_{\Pi_n}$ and $\mathbf{X}_{\Pi_n'}$ are uniformly distributed over the set of permutations of $[n]$ with probability 1. Also, since $\Pi_n$ and $\Pi_n'$ are independent, the ranks of $\mathbf{X}_{\Pi_n}$ and $\mathbf{X}_{\Pi_n'}$ are independent. Furthermore, note that, for each $i \in [n]$, 

$$
\hat{F}_{\Pi_n}(X_{\Pi_n(i)}) = \frac{1}{n} r(X_{\Pi_n(i)}) \, \, ,
$$

\noindent for $r(X_{\Pi_n(i)})$ the rank of $X_{\Pi_n(i)}$, i.e. $\hat{\sigma}_{\Pi_n}$ is also a rank statistic. In particular, it follows that, for 

$$
U_{\Pi_n} = U_n (\mathbf{X}_{\Pi_n}) \, \, ,
$$

\noindent and for $U_{\Pi_n '}$ defined analogously, $U_{\Pi_n}/\hat{\sigma}_{\Pi_n}$ and $U_{\Pi_n'}/\hat{\sigma}_{\Pi_n'}$ are i.i.d. random variables, each having the same distribution as 

$$
\frac{U_n(\tilde{X}_1, \, \dots, \, \tilde{X}_n)}{\hat{\sigma}_n (\tilde{X}_1, \, \dots, \, \tilde{X}_n) } \, \, ,
$$

\noindent where $\{ \tilde{X}_i: i \in [n]\}$ is a sequence of i.i.d. random variables, each having the marginal distribution $F$. It follows that, by (i), as $n \to \infty$, 

$$
\left( \frac{\sqrt{n}U_{\Pi_n}}{\hat{\sigma}_{\Pi_n} }, \, \frac{\sqrt{n}U_{\Pi_n'}}{\hat{\sigma}_{\Pi_n'} }\right) \overset{d}\to N(0, \, I_2) \, \, ,
$$

\noindent where $I_2$ is the $2 \times 2$ identity matrix. Therefore, by Theorem 15.2.3 of \cite{tsh}, the result of (ii) follows. \end{proof}

\begin{rem} \rm As a consequence of Theorem \ref{thm.perm.test}, a two-sided permutation test of the null hypothesis 

$$
H_0^{(g)}: \{X_n: n \in \nn\} \text{ is stationary,}
$$

\noindent which rejects for large values of $\sqrt{n}{U_n}/\hat{\sigma}_n$, is asymptotically valid under the stated conditions. \qed

\end{rem}

\begin{rem} \rm Since $\sqrt{n}U_n/\hat{\sigma}_n$ is a rank statistic, it follows that the permutation distribution $\hat{R}_n$ has no dependence on the underlying distribution of the $X_i$, as long as the marginal distribution $F$ is continuous. In particular, it follows that $\hat{R}_n$ retains a convenient property of the original null distribution of the Mann-Kendall test: namely, for a fixed choice of $n$ and the truncation parameter $b_n$ of the studentization factor, the permutation distribution is \textit{fixed}, and so may be computed in advance and tabulated. \qed 
\end{rem}

Having shown the asymptotic validity of the permutation test based on the studentized global Mann-Kendall statistic, we turn our attention to finding the local limiting power function of this test.

\begin{thm} Let $\{X_i: i \in [n]\}$ be a sequence of random variables satisfying the conditions of Theorem \ref{thm.null.conv}.  Suppose further that $F$ is twice differentiable, with density $f$ with respect to Lebesgue measure, such that $f$ has uniformly bounded first derivative $f'$, which is measurable (with respect to Lebesgue measure). Let $\{\mu_{n, \,i}: i \in [n]\} \subset \rr$ be such that, as $n \to \infty$,

\beq
\begin{aligned}
\frac{1}{\sqrt{n}} \sum_{i=1} ^n \mu_{n, \, i} ^2 &\to 0\, \, .
\end{aligned}
\label{alternative.mu.cond}
\eeq

\noindent Let 

$$
\nu_n = \frac{1}{\sqrt{n}}\sum_{i=1}^n \frac{n +1 -2i}{n-1} \mu_{n, \,i} \, \, .
$$

\noindent For each $i \in [n]$, let 

$$
Y_{n, \, i} = X_i + \mu_{n, \, i} \, \, .
$$ 

\noindent Let $\pp_n$, $\qq_n$ be the measures such that, for $A \in \mathcal{B}\left(\rr^n\right)$, 

$$
\begin{aligned}
\pp_n (A) &= \pp( (X_1, \, \dots, \, X_n) \in A) \\
\qq_n (A) &= \pp( (Y_{n, \,1}, \, \dots, \, Y_{n, \,n}) \in A) \, \, .
\end{aligned}
$$

\noindent Suppose that $\{\qq_n: n \in \nn\}$ is contiguous to $\{\pp_n: n \in \nn\}$. Then, under $\qq_n$, as $n\to \infty$,

$$
\sqrt{n}U_n + \nu_n \overset{d} \to N(0, \, \sigma^2) \, \, , 
$$

\noindent where $\sigma^2$ is as defined in Theorem \ref{thm.null.conv}.

\label{alternative.contig.lim} 
\end{thm}

\begin{rem}

\rm Note that the condition on $f'$ implies that $f$ is uniformly bounded. Indeed, suppose the contrary. Let $\lVert f' \rVert_\infty = k < \infty$. Then, for $N  > \sqrt{\frac{1}{2k} }$, for $x$ such that $f(x) \geq N + 1$,  

$$
\begin{aligned}
\int_x ^ {x + 2/N} f(t) \d t&\geq \int_x ^{x + 2/N}( f(x) -  k t) \d t \\
&= \frac{2f(x)}{N } - k \cdot \frac{2}{N^2} \\
&> 2 + \frac{2}{N} - 1 \\
&> 1 \,\, ,
\end{aligned} 
$$

\noindent i.e. we have obtained a contradiction. Note also that this implies that the random variable $f(X_1)$ is bounded, and so, in particular, has finite moments of every order. \qed

\label{f.is.bounded}
\end{rem}

\begin{cor} Under the conditions of Theorem \ref{alternative.contig.lim}, for $\hat{\sigma}_n ^2$ as defined in Theorem \ref{thm.consistency.var}, we have that, under $\qq_n$,

$$
\frac{\sqrt{n} U_n + \nu_n}{\hat{\sigma}_n} \overset{d} \to N(0, \, 1) \, \, .
$$

\label{cor.student.ctgs}
\end{cor}

\begin{proof} The result follows immediately from Theorem \ref{alternative.contig.lim}, Theorem \ref{thm.consistency.var}, the definition of contiguity, and Slutsky's theorem. \end{proof}

\begin{rem} \rm By Corollary \ref{cor.student.ctgs}, it follows that, as long as $\nu_n \to \nu \in [-\infty, \, \infty]$, for $\phi_n = \phi_n (Y_1, \, \dots,\, Y_n)$ the test function corresponding to the level $\alpha$ permutation test outlined in Theorem \ref{thm.perm.test}, the limiting power against the sequence of measures $\qq_n$ is given by 

$$
\ee_{\qq_n} \left[ \phi_n (Y_1, \, \dots, \, Y_n) \right] \to 1- \Phi\left(z_{1-\alpha}+\frac{\nu}{\sigma} \right) \, \, ,
$$

\noindent where $\Phi$ is the standard Gaussian c.d.f., $\sigma$ is as defined in (\ref{lim.variance.centre}), and $z_{1-\alpha}$ is the $(1-\alpha)$-quantile of the standard Gaussian distribution. \qed

\label{rem.loc.power}
\end{rem}

\begin{example} (White noise process) \rm Let $\{X_n: \, n \in \nn\}$ be a standard Gaussian white noise process, i.e. $X_n \sim N(0, \, 1)$. Let $h > 0$. For each $n \in \nn$, for each $i \in [n]$, let 

$$
Y_{n,\,i} = X_i + \frac{hi}{n^{3/2}} \, \, , 
$$

\noindent i.e. in the setting of Theorem \ref{alternative.contig.lim}, we take

$$
\mu_{n,\, i} = \frac{hi}{n^{3/2}} \, \, .
$$

\noindent Let $\pp_n$ and $\qq_n$ be as in Theorem \ref{alternative.contig.lim}. We begin by showing that $\{\qq_n\}$ is contiguous to $\{\pp_n\}$. We have that the log-likelihood ratio is given by 

$$
\log L_n = -\frac{h^2}{2n^3} \sum_{i=1}^n i^2 + \frac{h}{n^{3/2}} \sum_{i=1}^n i X_i \, \, .
$$

\noindent Let 

\beq
\mu_n = \frac{h^2}{6n^3} n (n+1) (2n+1) \, \, .
\label{lim.mean.ctgs}
\eeq

\noindent We have that 

$$
\log L_n \sim N\left( -\frac{1}{2} \mu_n, \, \mu_n ^2 \right) \, \, ,
$$

\noindent and so, by Corollary 12.3.1 of \cite{tsh}, $\{\qq_n\}$ and $\{\pp_n\}$ are mutually contiguous. Also, 

$$
\begin{aligned}
\frac{1}{\sqrt{n}} \sum_{i=1} ^n \left( \frac{hi}{n^{3/2}} \right)^2 &= \frac{h^2}{n^{7/2}} \sum_{i=1}^n i^2 \\
&= \frac{h^2 n(n+1)(2n+1)}{n^{7/2}} \\
&= o(1) \, \, .
\end{aligned}
$$

\noindent We may therefore apply Theorem \ref{alternative.contig.lim}, and so 

$$
\sqrt{n} U_n + \nu_n \overset{d} \to N\left(0, \, \frac{4}{9} \right) \, \, .
$$

\noindent By Remark \ref{rem.loc.power}, in order to compute the limiting power of the one-sided level $\alpha$ permutation test, it remains to compute the limit of the sequence $\{\nu_n\}$. We have that, as $n \to \infty$, 

$$
\begin{aligned}
\frac{1}{\sqrt{n}} \sum_{i=1} ^n \frac{n+1 -2i}{n-1} \cdot \frac{hi}{n^{3/2} } &= -\frac{h (n+1)}{6n} \\
&\to -\frac{h}{6} \, \, . 
\end{aligned}
$$

\noindent Therefore the local limiting power function of the permutation test is given by 

$$
\ee_{\qq_n} [\phi_n] \to 1- \Phi\left(z_{1-\alpha} -\frac{h}{4} \right)  \, \, , 
$$

\noindent where $\Phi$ is the standard Gaussian c.d.f. \qed 
\label{ex.iid.norm}
\end{example}

\begin{example} ($AR(1)$ process) \rm For $n \geq 1$, let $\{X_n: \, n \in \zz \}$ satisfy the equation, for $\rho \in \rr$ such that $|\rho| < 1$, 

\beq
X_{n + 1} = \rho X_n + \epsilon_{n+1}\, \, ,
\label{ar1.defn}
\eeq

\noindent where the $\epsilon_k$ are independent and identically distributed, with $\ee [\epsilon_k] = 0$, $\ee \left[\epsilon_k ^2\right] = 1$, and, for some $\delta > 0$, 

$$
\ee \left[ \left| \epsilon_k\right| ^{2 + \delta} \right] < \infty \, \, ,
$$ 

\noindent i.e. $X$ is an $AR(1)$ process. Since $|\rho| <1$, there exists a unique stationary solution to (\ref{ar1.defn}). By \cite{mokkadem}, Theorem 1, if the distribution of the $\epsilon_k$ is absolutely continuous with respect to Lebesgue measure, the conditions of Theorem \ref{thm.perm.test} are satisfied. Therefore, asymptotically, the rejection probability of the permutation test applied to such a sequence will be equal to the nominal level $\alpha$.

Now, consider the triangular array of sequences given by, for $i \in [n]$,

$$
Y_{n,\,i} = X_i + \mu_{n, \, i} \, \, ,
$$

\noindent where

$$ 
\mu_{n, \, i} = \frac{hi}{n^{3/2}} \, \, .
$$

\noindent Let $\pp_n$ and $\qq_n$ be as in Theorem \ref{alternative.contig.lim}. We begin by showing that $\{\qq_n\}$ is contiguous to $\{\pp_n\}$. We have that the log-likelihood ratio is given by 

$$
\begin{aligned}
\log L_n &= -\frac{1}{2\left(1- \rho^2 \right)} \sum_{i=1} ^n \mu_{n, \, i}^2 + \frac{1}{\left( 1- \rho^2 \right)} \sum_{i=1}^n \mu_{n, \, i} (X_i - \rho X_{i-1} ) \\
&= -\frac{1}{2\left(1- \rho^2 \right)} \sum_{i=1} ^n \mu_{n, \, i}^2  + \frac{1}{\left( 1- \rho^2 \right)} \sum_{i=1}^n\mu_{n,\, i} \epsilon_i \, \, .
\end{aligned}
$$

\noindent Let $\mu_n$ be as in (\ref{lim.mean.ctgs}). We have that 

$$
\begin{aligned}
\var \left( \sum_{i=1}^n \mu_{n,\,i} \epsilon_i \right) &= \mu_n ^2 \\
&= \frac{h^2}{3} \left( 1 + o(1) \right) \, \, .
\end{aligned}
$$

\noindent Also, by an integral approximation, 

$$
\begin{aligned}
\sum_{i=1} ^n \ee \left[ \left|\mu_{n, \, i} \epsilon_{i } \right| ^{2 + \delta} \right] &= \ee \left[ | \epsilon_{1} |^{2 + \delta} \right] \cdot \frac{h^{2+ \delta}}{3n^{3(2 + \delta)/2} } \sum_{i=1} ^n i^{2 + \delta} \\
&= \ee \left[ | \epsilon_{1} |^{2 + \delta} \right] \cdot \frac{h^{2+ \delta} \cdot n^{3 + \delta}}{3n^{3 + 3\delta/2} } \cdot \frac{1}{3 + \delta} (1 + o(1)) \\
&= o(1) \, \, .
\end{aligned}
$$

\noindent In particular, it follows that the triangular array 

$$
Z_{n, \, i} = \mu_{n,\, i} \epsilon_i 
$$

\noindent satisfies the Lyapunov condition, and so, by the Lyapunov central limit theorem and Slutsky's theorem,

$$
\log L_n \overset{d} \to N\left(-\frac{h^2}{6}, \, \frac{h^2}{3} \right) \, \, .
$$

\noindent Therefore, by Corollary 12.3.1 of \cite{tsh}, $\{\qq_n\}$ and $\{\pp_n\}$ are mutually contiguous. Hence, by the same argument as in Example \ref{ex.iid.norm}, as long as the marginal distribution of the $X_i$ satisfies the conditions of Theorem \ref{alternative.contig.lim}, the local limiting power of the one-sided level $\alpha$ permutation test is given by 

$$
\ee_{\qq_n} [ \phi_n] \to 1- \Phi\left(z_{1-\alpha} - \frac{h}{6 \sigma} \right) \, \, ,
$$

\noindent where $\sigma$ is as in (\ref{lim.variance.centre}). \qed

\label{ex.global.mk.power}
\end{example} 

\section{Tests of local trend}\label{sec.local}

In this section, we discuss the problem of conducting a hypothesis test of local trend, i.e. of testing the null hypothesis, for a stationary sequence $\{X_n: n \in \nn\}$ and some $M \in \nn$, 

$$
H^{(l)} _{0, \, M}: \frac{1}{(n-M)M} \sum_{i=1}^{n-M} \sum_{j={i+1}} ^{i+M } \left( \pp(X_i < X_j) - \pp(X_i > X_j) \right) = 0 \, \, .
$$

Note that, if $H_{0,\, M}^{(l)}$ holds for all $M \in \nn$, we have that, for all $i \neq j$, 

$$
\pp(X_i < X_j) = \pp(X_i > X_j) \, \, ,
$$

\noindent i.e. the intersection of this countable family of null hypotheses results in a stronger condition on the sequence $\{X_n: n \in \nn\}$ than $H_0 ^{(g)}$.

In order to construct an appropriate permutation test of $H_{0,\, M}^{(l)}$, we begin by defining a family of statistics, which we term the \textit{local} Mann-Kendall statistics.

\subsection{The local Mann-Kendall statistic}\label{sec.partial}

\begin{defn} \rm Let $n \in \nn$, and let $\{X_i: i \in [n]\}$ be a sequence of random variables. Let $g(n) \in [n-1]$. Let 

$$
V_n = V_n(X_1,\, \dots,\, X_n)=  \frac{1}{n g(n) } \sum_{i=1}^{n - g(n)} \sum_{j=i+1} ^{i + g(n) } \left( I\{X_i < X_j\} - I\{X_i > X_j\} \right) 
$$

\noindent be the local Mann-Kendall statistic of order $g(n)$.

\end{defn}

\begin{rem} \rm A heuristic interpretation of the parameter $g(n)$ is that it functions as a choice of local bandwidth, i.e. a choice of larger values of $g(n)$ will increase the consideration of the ordering of values of the $X_i$ that are further apart in time.
\end{rem}

While one may define the local Mann-Kendall statistic of order $g(n)$ for any value of $g(n)$ less than $(n-1)$, its primary use throughout this section will be to test hypotheses of a lack of local monotone trend. In general, the two choices of $g:\nn\to \nn$ under consideration throughout the remainder of this section will be either $g(n) = M$ for all $n \in \nn$, or $g$ being a nondecreasing function of $n$ such that $g(n)/n \to 0$ as $n \to \infty$.

With the definition of the local Mann-Kendall statistic in hand, we may begin consideration of a permutation test based on $V_n$. Under the additional assumption that the sequence $\{X_i: i \in [n]\}$ is i.i.d., the randomization hypothesis holds, and so the permutation test based on any test statistic will be exact in finite samples at the nominal level $\alpha$. However, in a weakly dependent setting, the randomization hypothesis does not hold, and so the permutation test based on may not be exact in finite samples or even asymptotically valid. In order to commence appropriately assessment of the validity of such a test, we provide the following result, which describes the limiting permutation distribution based on the test statistic $\sqrt{ng(n)} V_n$.

\begin{thm} Let $\{X_i: i \in \nn\}$ be a sequence of random variables such that, for all $i \neq j$, $\pp(X_i =X_j ) = 0$. Let $g: \nn \to \nn$ be such that $g(n) = o\left( n^{1/7} \right)$.

Let $\hat{R}_n$ be the permutation distribution, based on the test statistic $\sqrt{ng(n)} V_n$, with associated group of transformations $S_n$, the symmetric group of order $n$. For each $t \in \rr$, we have that, as $n \to \infty$,

$$
\sup_{t \in \rr} \left| \hat{R}_n (t) - \Phi(t\sqrt{3}) \right| \to 0 \ \, \, ,
$$

\noindent where $\Phi$ is the distribution of a standard Gaussian random variable.

\label{thm.perm.unstud.local}
\end{thm}

\begin{example}\label{ex.local.counterex} ($MA(2), \, g(n) \equiv 1$) \rm Consider the setting in which $g(n) = 1$ for all $n \in \nn$, and, for all $i$, 

$$
X_i = \phi_0 \epsilon _i + \phi_1 \epsilon_{i-1} \, \, ,
$$

\noindent where the $\epsilon_i$ are i.i.d. random variables with a symmetric continuous distribution $F$. We have that 

$$
V_n = \frac{1}{n}\sum_{i=1} ^{n-1} \left( I\{ X_{i+1} > X_i \} - I\{X_{i+1} < X_i \} \right) \, \, .
$$

\noindent Let

$$
\begin{aligned}
p_1 &= \pp\left( X_{3} > X_2 > X_1 \right) \\
p_2 &=  \pp\left( X_{3} < X_2 < X_1 \right) \, \, ,
\end{aligned}
$$

\noindent and let $r = 2(p_1 + p_2) - 1$. We have that 

$$
\begin{aligned}
&\var(V_n)\\
&= \frac{1}{n^2} \left( n-1 + 2\sum_{i=1}^{n-2} \cov\left( I\{ X_{i+2} > X_{i+1} \} -I\{ X_{i+2} < X_{i+1} \} , \, I\{ X_{i+1} > X_i \} -I\{ X_{i+1} < X_i \} \right) \right) \\
&= \frac{1}{n^2}\left( n- 1 + 2r(n-2) \right) \\
&= \frac{1+2r}{n} + o\left (\frac{1}{n} \right) \, \, .
\end{aligned}
$$

\noindent It follows that the limiting variance of $\sqrt{ng(n)} V_n$ is given by

$$
\lim_{n \to \infty} \var\left( \sqrt{n g(n)} V_n \right) = 1 + 2r \, \, .
$$

\noindent In the setting in which the $X_i$ are i.i.d. with marginal distribution $F$, we have that $p_1 = p_2 = 1/6$, and so $r = -1/3$. However, since $r$ is, in general, not equal to $-1/3$ in the setting described above, this is not equal to the variance computed in Theorem \ref{thm.perm.unstud.local}. \qed

\end{example}

As illustrated by Example \ref{ex.local.counterex},  a permutation test based on $\sqrt{ng(n)} V_n$ will not be asymptotically valid, since the limiting variance of the permutation distribution and the limiting variance of the test statistic will not be equal in general. We may however proceed as usual, and find an appropriate estimator of the limiting variance of $V_n$ which is asymptotically consistent, and which, under the action of a random permutation, converges to 1 in probability.

\begin{lem} Let $\{X_i, \, i \in[n] \}$ be a stationary, $\alpha$-mixing sequence such that, for all $i \neq j$, $\pp(X_i = X_j) =0$. Suppose also that 

$$
\sum_{n \geq 1} \alpha_X (n) < \infty \, \, .
$$

\noindent Let $g(n) \equiv G \in \nn$ be constant. For each $i \in [n]$, let 

$$
Y_i = \sum_{j = (i - G) \vee 1 }^{i-1} \left( I\{X_i > X_j\} - I\{X_i < X_j\} \right) \, \, .
$$

\noindent Let $b_n \in \nn$ be such that $b_n = o(\sqrt{n})$ and, as $n \to \infty$, $b_n \to \infty$. Let 

\beq
\hat{\sigma}_n^2 = \frac{1}{n} \sum_{i=1} ^n \left(Y_i - \bar{Y}_n\right) ^2 + \frac{2}{n} \sum_{j=1} ^{b_n} \sum_{i=1}^{n-j} (Y_j - \bar{Y}_n) (Y_{i+j} - \bar{Y}_n) \, \, .
\label{sigma.n.lmk}
\eeq

\noindent Let 

$$
\sigma^2 = \var(Y_{G+1}) + 2 \sum_{k \geq 1} \cov(Y_{G+1}, \, Y_{G+ k + 1} ) \, \, .
$$

\noindent We have that $\sigma^2 < \infty$, and, as $n \to \infty$,

$$
\hat{\sigma}_n^2 \overset{p} \to \sigma^2 \, \, .
$$

\label{lem.var.stud}
\end{lem}

With the result of Lemma \ref{lem.var.stud} in hand, we may now proceed to studentize the test statistic. By an application of Slutsky's theorem for randomization distributions, this studentization factor will have no effect on the limiting permutation distribution, but will result in the limiting distribution of the test statistic being asymptotically pivotal. Therefore, combining the previous results, we obtain the following.

\begin{thm} In the setting of Lemma \ref{lem.var.stud}, let $\mu = \ee Y_{1+G} $. Let 

$$
\hat{\tau}_n ^2 = \frac{1}{G} \hat{\sigma}_n ^2 \,\,,
$$

\noindent and let $\nu = \mu/G$. We have the following:

\begin{enumerate}

\item[(i)] As $n \to \infty$, 

$$
\frac{\sqrt{nG}(V_n - \nu )}{\hat{\tau}_n} \overset{d} \to N(0, \, 1) \, \, .
$$

\item[(ii)] Let $\hat{R}_n$ be the permutation distribution, with associated group of transformations $S_n$, the symmetric group of order $n$, based on the test statistic $\sqrt{nG} V_n/\hat{\tau}_n$. We have that, as $n\to\infty$, 

$$
\sup_{t \in \rr} \left| \hat{R}_n (t) - \Phi(t) \right| \overset{a.s.}\to 0 \, \,, 
$$

\noindent where $\Phi$ is the distribution of a standard Gaussian random variable.

\end{enumerate} 
\label{local.perm.thm}
\end{thm}

\begin{proof} As in the proof of Lemma \ref{lem.var.stud}, let $Z_i = Y_{i +G}$, for $1 \leq i \leq n -G$. Let $m = n-G$. Since the $Y_i$ are uniformly bounded, we have that 

$$
\begin{aligned}
\sqrt{n} (\bar{Y}_n-\mu) &= \sqrt{m}( \bar{Z}_m -\mu) (1 + o(1)) + \frac{1}{\sqrt{n}} \sum_{i=1} ^G Y_i \\
&= \sqrt{m}( \bar{Z}_m -\mu) (1 + o(1)) + O\left(\frac{1}{\sqrt{n}} \right) \, \, .
\end{aligned}
$$

\noindent In particular, by Ibragimov's central limit theorem \citep{ibragimov}, Slutsky's theorem, and the result of Lemma \ref{lem.var.stud}, we have that

$$
\frac{\sqrt{n}\left( \bar{Y}_n - \mu \right)}{\hat{\sigma}_n} \overset{d}\to N(0, \, 1)\,\, .
$$

\noindent Since 

$$
V_n = \frac{1}{g(n)} \bar{Y}_n \, \, ,
$$

\noindent the result of (i) follows.

Turning our attention to (ii), we observe that, since each $Y_i$ is a rank statistic, and, conditional on the data, under the action of $\Pi_n \sim \text{Unif}(S_n)$, the ranks of the $X_i$ are uniformly distributed with probability 1, the permutation distribution $\hat{R}_n$ is exactly equal to the unconditional distribution of the test statistic $\sqrt{ng(n)}V_n /\hat{\tau}_n$ in the case when the $X_i$ are i.i.d. with a continuous marginal distribution $F$. In this case, since $\mu = 0$, the result of (ii) follows from that of (i). 
\end{proof}

\begin{rem} \rm As a result of Theorem \ref{local.perm.thm}, under the conditions laid out therein we have that a one-sided permutation test of the null hypothesis $H_{0,\, M}^{(l)}$, based on the test statistic $\sqrt{nM}V_n/\hat{\tau}_n$, where $V_n$ is the local Mann-Kendall statistic of order $M$, is asymptotically valid.

\end{rem} 

\begin{rem} \rm Let $M, \, r \in \nn$ such that $M > r$. By an entirely analogous argument to the one presented in this section, one may construct permutation tests of the null hypothesis 

$$
H_{0, \, M,\, r} ^{(l)} \frac{1}{(n-M) \binom{M}{r} }\sum_{ i_1 < i_2 < \dots < i_r \leq i_1 + M} \left(\pp(X_{i_1} < \dots < X_{i_r} ) - \pp(X_{i_1} > \dots > X_{i_r} ) \right) = 0 \, \, ,
$$

\noindent using a studentized version of the test statistic

$$
\frac{1}{(n-M)\binom{M}{r} } \sum_{ i_1 < i_2 < \dots < i_r \leq i_1 + M} \left(I(X_{i_1} < \dots < X_{i_r} ) - I(X_{i_1} > \dots > X_{i_r} ) \right) \,\,.
$$
\end{rem} 

\section{Simulation results}\label{sec.simulations.monotone.trend}

In this section, we provide Monte Carlo simulations illustrating our results. Subsection \ref{subsec.global.sim} provides simulation results of the performance of the global Mann-Kendall permutation test, and Subsection \ref{subsec.local.sim} illustrates the performance of the local Mann-Kendall permutation test. In both sections, we consider one-sided tests, i.e. those rejecting for large values of the test statistic, conducted at the nominal level $\alpha = 0.05$. The simulation results confirm that the two permutation tests are valid in that, in large samples, the rejection probability under the null hypothesis is approximately equal to $\alpha$.

\subsection{Testing for global trend}\label{subsec.global.sim}

We begin with a comparison of the performance of the studentized global Mann-Kendall permutation test against the classical Mann-Kendall test. As a review, the classical Mann-Kendall test computes the statistic 

$$
U_n = U_n(X_1, \, \dots,\, X_n) ={\binom{n}{2}}^{-1} \sum_{i < j} \left( I\{X_i < X_j\} - I\{X_i > X_j\}\right) \, \, .
$$

\noindent Under the null hypothesis that the sequence $\{X_n: n \in \nn\}$ is i.i.d. and that $\pp(X_i = X_j) = 0$ for all $i \neq j$, the statistic $U_n$ has a \textit{fixed} distribution $G_n$, and so the test compares $U_n$ to the $(1- \alpha)$ quantile of $G_n$ in order to reject or not reject the null hypothesis. 

Note that, by the result of Theorem \ref{basic.perm.lim}, the permutation distribution based on the unstudentized Mann-Kendall statistic converges weakly in probability to the same limiting distribution as the Mann-Kendall statistic under the additional assumption that the sequence is i.i.d. It follows that, for a fixed infinite-dimensional sequence $\{X_n: n \in \nn\}\sim P$, the rejection probability of the classical Mann-Kendall test and the unstudentized global Mann-Kendall permutation test will be asymptotically equal as $n \to \infty$. On account of this, in the following simulations, we present a comparison of the studentized global Mann-Kendall permutation test and the classical Mann-Kendall test.

In this simulation, we consider the following two settings: in Table \ref{tab.mdep.gmk}, we consider processes of the following form. Let $m \in \nn$. For $\{Z_n: n \in \nn\}$ be i.i.d. standard Gaussian random variables, for each $n \in \nn$, let

$$
X_n = \prod_{j=0}^{m -1} Z_{n + j} \, \, .
$$

\noindent This sequence is $m$-dependent and stationary, and, for each $i \neq j$, $\pp(X_i \neq X_j) = 0$. Therefore $\{X_n: n \in \nn\}$ satisfies the conditions of Theorem \ref{thm.perm.test}, and so the corresponding permutation test is asymptotically valid at the nominal level $\alpha$. Several distributions other than the standard Gaussian distribution were considered for the distribution of the $Z_i$, but these were observed to have little impact on the rejection probabilities of either the studentized global Mann-Kendall permutation test or the classical Mann-Kendall test.

\begin{table}[h]
\centering % centering table
\begin{tabular}{cc rrrrrrr} % creating eight columns
\hline\hline %inserting double-linex
$m$&$n$&10  &50&    100  &  500 &  1000   \\ [0.5ex]
\hline % inserts single-line
\multirow{2}{*}{0}	&Stud. Perm.&0.045 &0.059&0.047&0.047&0.054\\
				&Classical M-K &0.042&0.038&0.047&0.049&0.044\\
				 \hline
\multirow{2}{*}{10}&Stud. Perm.& 0.054&0.061&0.062&0.057&0.055\\
				&Classical M-K&0.059&0.061&0.076&0.036&0.055 \\
				 \hline
\multirow{2}{*}{20}&Stud. Perm.&0.036&0.082&0.063&0.051&0.045\\
				&Classical M-K &0.056&0.062&0.069&0.057&0.061 \\
				 \hline
\multirow{2}{*}{30}	&Stud. Perm.&0.056&0.085&0.077&0.052&0.046\\
				&Classical M-K &0.066&0.083&0.071&0.052&0.060 \\
\hline \hline% inserts single-line
\end{tabular}
\caption{Monte Carlo simulation results for null rejection probabilities for tests of $H_0^{(g)}$, in an $m$-dependent Gaussian product setting.} \label{tab.mdep.gmk}
\end{table}

We also consider more traditional autoregressive settings in Tables \ref{tab.ar1.gmk} and \ref{tab.ar2.gmk}. Namely, we consider the following two examples. In Table \ref{tab.ar1.gmk}, we consider the following processes: for $\{\epsilon_n: n \in \nn\}$ a sequence of i.i.d. standard Gaussian random variables and some constant $\rho \in (-1,\, 1)$, let $X_1 \sim N(0, \, (1-\rho^2)^{-1})$, independent of the $\epsilon_i$, and for each $n \geq 2$, let 

$$
X_n = \rho X_{n -1} + \epsilon_n \,\, ,
$$

\noindent i.e. $\{X_n: n \in \nn\}$ follows the unique distribution of the stationary $AR(1)$ process with standard Gaussian innovations. 

\begin{table}[h]
\centering % centering table
\begin{tabular}{cc rrrrrrr} % creating eight columns
\hline\hline %inserting double-linex
$\rho$&$n$&10  &50&    100  &  500 &  1000   \\ [0.5ex]
\hline % inserts single-line
\multirow{2}{*}{-0.6}	&Stud. Perm.&0.047 &0.192&0.039&0.074&0.043\\
				&Classical M-K &0.004&0.002&0.002&0.003&0.000\\
				 \hline
\multirow{2}{*}{-0.2}&Stud. Perm.& 0.043&0.072&0.058&0.046&0.057\\
				&Classical M-K&0.029&0.028&0.019&0.021&0.025 \\
				 \hline
\multirow{2}{*}{0.2}&Stud. Perm.&0.030&0.053&0.046&0.042&0.058\\
				&Classical M-K &0.081&0.081&0.084&0.085&0.099 \\
				 \hline
\multirow{2}{*}{0.6}&Stud. Perm.&0.035&0.052&0.068&0.042&0.049\\
				&Classical M-K &0.183&0.200&0.196&0.213&0.206 \\
\hline \hline% inserts single-line
\end{tabular}
\caption{Monte Carlo simulation results for null rejection probabilities for tests of $H_0^{(g)}$, in an $AR(1)$ setting.} \label{tab.ar1.gmk}
\end{table}

Similarly, in Table \ref{tab.ar2.gmk}, we consider processes of the following form: for $\rho \in (-1,\, 1)$, let $\{Z_n: n \in \nn\}$ and $\{Z_n': n \in \nn\}$ be independent stationary $AR(1)$ processes with i.i.d. standard Gaussian innovations. For $n \in \nn\cup \{0\}$, let 

$$
\begin{aligned}
X_{2 n + 1} &= Z_n \\
X_{2n + 2} &= Z_n' \,\, .
\end{aligned}
$$

\noindent Note that we may also express $\{X_n\}$ as the unique stationary process satisfying the autoregressive equation, for $n \geq 3$,

$$
X_n = \rho X_{n-2} + \eta_n \, \, ,
$$

\noindent for $\{\eta_n: n \in \nn\}$ a sequence of i.i.d. standard Gaussian random variables, i.e. the $X_i$ follow an $AR(2)$ process.

As discussed in Example \ref{ex.global.mk.power}, the asymptotic rejection probability of the permutation test applied to such sequences is equal to the nominal level $\alpha$. 

For each of the three situations, 1,000 simulations were performed. Within each simulation, the permutation test was calculated by randomly sampling 1,000 permutations.

\begin{table}[h]
\centering % centering table
\begin{tabular}{cc rrrrrrr} % creating eight columns
\hline\hline %inserting double-linex
$\rho$&$n$&10  &50&    100  &  500 &  1000   \\ [0.5ex]
\hline % inserts single-line
\multirow{2}{*}{-0.6}&Stud. Perm.&0.148 &0.354&0.024&0.188&0.033\\
				&Classical M-K &0.018&0.003&0.002&0.000&0.000\\
				 \hline
\multirow{2}{*}{-0.2}&Stud. Perm.& 0.074&0.092&0.059&0.052&0.039\\
				&Classical M-K&0.025&0.023&0.034&0.029&0.026 \\
				 \hline
\multirow{2}{*}{0.2}&Stud. Perm.&0.030&0.048&0.048&0.037&0.055\\
				&Classical M-K &0.054&0.086&0.096&0.081&0.093 \\
				 \hline
\multirow{2}{*}{0.6}&Stud. Perm.&0.009&0.104&0.060&0.059&0.067\\
				&Classical M-K &0.074&0.179&0.192&0.195&0.192 \\
\hline \hline% inserts single-line
\end{tabular}
\caption{Monte Carlo simulation results for null rejection probabilities for tests of $H_0^{(g)}$, in an $AR(2)$ setting.} \label{tab.ar2.gmk}
\end{table}

We observe that, in the $m$-dependent Cauchy product setting of Table \ref{tab.mdep.gmk}, both the studentized global Mann-Kendall permutation test and the classical Mann-Kendall test perform comparably, both controlling the rejection probability at (close to) the nominal level $\alpha$. However, in contrast, while the studentized global Mann-Kendall permutation test also exhibits Type 1 error control at the nominal level $\alpha$ in both the $AR(1)$ and $AR(2)$ settings of Tables \ref{tab.ar1.gmk} and \ref{tab.ar2.gmk}, respectively, we observe the following phenomena. For $\rho > 0$, we observe that the rejection probability of the classical Mann-Kendall test is significantly greater than $\alpha$, i.e. we do not have Type 1 error control. In addition, for $\rho < 0$, we observe that the performance of the classical Mann-Kendall test is also unsatisfactory; namely, the rejection probabilities obtained are significantly below the nominal level $\alpha$, i.e. the test is overly conservative.

These issues may be explained as follows: since the limiting variance $\sigma^2$ of the Mann-Kendall statistic is given by (\ref{lim.variance.centre}), for positively (negatively) autocorrelated sequences we have that $\sigma^2$ is greater than (less than) the limiting variance of the Mann-Kendall statistic under the additional assumption that the sequence is i.i.d. Heuristically, we have the interpretation that the classical Mann-Kendall test ``confuses" positive or negative autocorrelation with positive and negative trend, respectively, whereas the studentized global Mann-Kendall test does not succumb to this issue.

We observe several computational choices to be made when applying the permutation testing framework in practice. By the results of Theorems \ref{thm.consistency.var} and \ref{thm.perm.test}, for large values of $n$, the estimate $\hat{\sigma}_n^2$, as defined in (\ref{var.est.global}), will be be strictly positive with high probability. However, for smaller values of $n$, it may be the case that a numerically negative value of $\hat{\sigma}_n ^2$ is observed, either when computing the test statistic or the permutation distribution. A trivial solution to this issue is the truncate the estimate at some sufficiently small fixed lower bound $\epsilon >0$. Note that, for appropriately small choices of $\epsilon$, i.e. $\epsilon < \sigma^2$, the results of Theorems \ref{thm.consistency.var} and \ref{thm.perm.test} still hold, i.e. inference based on this choice of studentization is still asymptotically valid. In practice, however, the suitability of a choice of $\epsilon$ for a particular numerical application is affected by the distribution of the $X_i$, and, in general, the estimated variance $\hat{\sigma}_n^2$ is bounded away from zero with high probability, on account of the additive constant $4/9$ in the expression (\ref{var.est.global}). For the above simulation, a constant value of $\epsilon = 10^{-3}$ was used.

A further choice is that of the truncation sequence $\{b_n: \, n \in \nn\}$ used in the definition of $\hat{\sigma}_n ^2$. Any sequence $  \{b_n\}$ such that, as $n \to \infty$, $b_n \to \infty$ and $b_n = o\left(\sqrt{n} \right)$ is theoretically justified by Theorem \ref{thm.perm.test}, although, in a specific setting, some choices of $\{b_n\}$ will lead to more numerical stability than others. In practice, several choices of $\{b_n\}$ were implemented, but were found to make little difference to the rejection probabilities observed. In the simulations above, $\{b_n\}$ was taken to be $b_n = [n^{1/3}]$, where $[x]$ denotes the integer part of $x$.

\subsection{Tests of local trend}\label{subsec.local.sim}

We now turn our attention to simulations involving the local Mann-Kendall permutation test of the hypothesis $H_{0, \, M}^{(l)}$. Under the assumption that the sequence $\{X_i: i \in [n]\}$ is i.i.d., the randomization hypothesis holds, and so the unstudentized local Mann-Kendall permutation test will be exact in finite samples and asymptotically valid. However, in general, the randomization hypothesis does not hold under $H_{0,\,M }^{(l)}$, and so the unstudentized local Mann-Kendall permutation test will not be exact or even asymptotically valid. However, by the result of Theorem \ref{local.perm.thm}, the studentized local Mann-Kendall permutation test will be asymptotically valid at the nominal level $\alpha$. 

In order to illlustrate this behavior, we provide a comparison of the simulated rejection probabilities of the studentized and unstudentized local Mann-Kendall permutation tests in two different settings. In both of the settings described below, we consider the choice of $M =5$. For both of the following situations, 1,000 simulations were performed. Within each simulation, the permutation test was calculated by randomly sampling 1,000 permutations. 

In Table \ref{tab.mdep.lmk}, we compare the simulated rejection probabilities of these two tests in the Gaussian product $m$-dependent setting described in Subsection \ref{subsec.global.sim}.

\begin{table}[h]
\centering % centering table
\begin{tabular}{cc rrrrrrr} % creating eight columns
\hline\hline %inserting double-linex
$m$&$n$&20  &50&    100  &  500 &  1000   \\ [0.5ex]
\hline % inserts single-line
\multirow{2}{*}{0}	&Stud. Perm.&0.049 &0.047&0.051&0.046&0.058\\
				&Unstud. Perm.&0.060&0.062&0.061&0.057&0.058\\
				 \hline
\multirow{2}{*}{1}&Stud. Perm.& 0.050&0.066&0.050&0.023&0.032\\
				&Unstud. Perm.&0.099&0.096&0.098&0.091&0.101 \\
				 \hline
\multirow{2}{*}{2}&Stud. Perm.&0.057&0.078&0.065&0.009&0.021\\
				&Unstud. Perm.&0.115&0.129&0.131&0.164&0.138 \\
				 \hline
\multirow{2}{*}{3}	&Stud. Perm.&0.045&0.097&0.064&0.001&0.019\\
				&Unstud. Perm.&0.141&0.160&0.167&0.150&0.180 \\
\hline \hline% inserts single-line
\end{tabular}
\caption{Monte Carlo simulation results for null rejection probabilities for tests of $H_{0,\,M}^{(l)}$, for $M = 5$, in an $m$-dependent Gaussian product setting.} \label{tab.mdep.lmk}
\end{table}

We observe that, in Table \ref{tab.mdep.lmk}, the studentized permutation test controls the rejection probability at the nominal level $\alpha$. However, in contrast, we observe that the unstudentized permutation test only has a rejection probability approximately equal to $\alpha$ in the case $m = 0$, i.e. when the $X_i$ are i.i.d. and the randomization hypothesis holds. For $m > 0$, the unstudentized permutation test visibly does not control the rejection probability at the nominal level $\alpha$.

In Table \ref{tab.ar1.lmk}, we provide a comparison of the rejection probabilities of the studentized and unstudentized local Mann-Kendall permutation tests in the $AR(1)$ setting of Subsection \ref{subsec.global.sim}.

\begin{table}[h]
\centering % centering table
\begin{tabular}{cc rrrrrrr} % creating eight columns
\hline\hline %inserting double-linex
$\rho$&$n$&20  &50&    100  &  500 &  1000   \\ [0.5ex]
\hline % inserts single-line
\multirow{2}{*}{-0.6}	&Stud. Perm.&0.048 &0.182&0.050&0.055&0.044\\
				&Unstud. Perm.&0.021&0.028&0.036&0.049&0.042\\
				 \hline
\multirow{2}{*}{-0.2}&Stud. Perm.& 0.062&0.070&0.053&0.038&0.045\\
				&Unstud. Perm.&0.049&0.052&0.039&0.047&0.052 \\
				 \hline
\multirow{2}{*}{0.2}&Stud. Perm.&0.037&0.029&0.033&0.078&0.061\\
				&Unstud. Perm.&0.085&0.072&0.071&0.087&0.066 \\
				 \hline
\multirow{2}{*}{0.6}	&Stud. Perm.&0.015&0.024&0.007&0.009&0.030\\
				&Unstud. Perm.&0.177&0.145&0.125&0.111&0.114 \\
\hline \hline% inserts single-line
\end{tabular}
\caption{Monte Carlo simulation results for null rejection probabilities for tests of $H_{0,\,M}^{(l)}$, for $M = 5$, in an $AR(1)$ setting.} \label{tab.ar1.lmk}
\end{table}

We observe that, while the unstudentized permutation test somewhat surprisingly provides Type 1 error control at the nominal level $\alpha$ for negative values of $\rho$, it does not do so for positive values of $\rho$. For all values of $\rho$, the studentized permutation test has rejection probabilities below the nominal level $\alpha$.

As in Subsection \ref{subsec.global.sim}, there are several computational choices to be made in practice. In particular, on account of the same numerical issues involving the studentization factor $\hat{\tau}_n^2$, we truncate the variance estimate at the lower bound $\epsilon = 10^{-3}$. Similarly, we must choose a truncation sequence $\{b_n: n \in \nn\}$ to be used in the definition of $\hat{\tau}_n^2$. In the above simulations, the choice $b_n = [n^{1/3}]$ was used. 

\section{Conclusions}

When the fundamental assumption of exchangeability does not necessarily hold, permutation tests are invalid unless strict conditions on underlying parameters of the problem are satisfied. For instance, the permutation test of $H_0^{(g)}$ based on the classical Mann-Kendall statistic is asymptotically valid only when $\sigma^2$ is as defined in (\ref{lim.variance.centre}), is equal to 4/9. Hence rejecting the null must be interpreted correctly, since rejection of the null with this permutation test does not necessarily imply that the sequence truly does exhibit monotone trend, in the sense that the quantity 

$$
\lim_{n \to \infty}\binom{n}{2}^{-1}\sum_{i < j}(\pp(X_i < X_j) - \pp(X_i > X_j))
$$

\noindent may be equal to zero, and the Mann-Kendall test will still reject the null hypothesis. We provide a testing procedure that allows one to obtain asymptotic rejection probability $\alpha$ in a permutation test setting. A significant advantage of this test is that it retains the property of finite-sample exactness of the Mann-Kendall test under the assumption of i.i.d., as well as achieving asymptotic level $\alpha$ in a much wider range of settings than the aforementioned tests. This test also retains the convenient property of the classical Mann-Kendall test that the permutation distribution is fixed for a given sample size $n$. An analogous permutation testing procedure also permits for asymptotically valid inference for the newly-introduced notion of local monotone trend.

Correct implementation of a permutation test is crucial if one is interested in confirmatory inference via hypothesis testing; indeed, proper error control of Type 1, 2 and 3 errors can be obtained for tests of global or local monotone trend by basing inference on test statistics which are asymptotically pivotal. A framework has been provided for tests of both local and global monotone trend.

In this paper, we have  defined specific  notions of a lack of global monotone trend and a lack of local monotone trend of order $M$, and constructed asymptotically valid testing procedures of these null hypotheses. Future work will expand on the methods presented herein, in order to provide analogous permutation testing procedures in the context of other commonly-used tests for monotone trend.

\addcontentsline{toc}{section}{Bibliography}

\bibliographystyle{apalike}
\bibliography{perm_test_thesis}

\begin{thebibliography}{}

\bibitem[Bradley, 2005]{bradley2005}
Bradley, R.~C. (2005).
\newblock Basic properties of strong mixing conditions. a survey and some open
  questions.
\newblock {\em Probab. Surveys}, 2:107--144.

\bibitem[Dietz and Killeen, 1981]{dietz_killeen}
Dietz, E.~J. and Killeen, T.~J. (1981).
\newblock A nonparametric multivariate test for monotone trend with
  pharamaceutical applications.
\newblock {\em Journal of the American Statistical Association},
  76(373):169--174.

\bibitem[Freedman, 2011]{freedman}
Freedman, D. (2011).
\newblock {\em Markov Chains}.
\newblock Springer, New York, NY.

\bibitem[Hamed, 2008]{hamed_2008}
Hamed, K.~H. (2008).
\newblock Trend detection in hydrologic data: The {Mann–Kendall} trend test
  under the scaling hypothesis.
\newblock {\em Journal of Hydrology}, 349(3):350--363.

\bibitem[Han and Qian, 2018]{hanasymptotics}
Han, F. and Qian, T. (2018).
\newblock {On inference validity of weighted U-statistics under data
  heterogeneity}.
\newblock {\em Electronic Journal of Statistics}, 12(2):2637 -- 2708.

\bibitem[Hoeffding, 1948]{hoeffding}
Hoeffding, W. (1948).
\newblock A class of statistics with asymptotically normal distribution.
\newblock {\em The Annals of Mathematical Statistics}, 19(3):293 -- 325.

\bibitem[Ibragimov, 1962]{ibragimov}
Ibragimov, I.~A. (1962).
\newblock Some limit theorems for stationary processes.
\newblock {\em Theory of Probability \& Its Applications}, 7(4):349--382.

\bibitem[Kendall, 1990]{kendall}
Kendall, M.~G. (1990).
\newblock {\em Rank Correlation Methods}.
\newblock Charles Griffin Book. Oxford University Press, London, England, 5
  edition.

\bibitem[Lehmann and Romano, 2022]{tsh}
Lehmann, E.~L. and Romano, J.~P. (2022).
\newblock {\em Testing Statistical Hypotheses}.
\newblock Springer texts in statistics. Springer Nature, 4 edition.

\bibitem[Mann, 1945]{mann1945}
Mann, H.~B. (1945).
\newblock Nonparametric tests against trend.
\newblock {\em Econometrica}, 13(3):245--259.

\bibitem[Mokkadem, 1988]{mokkadem}
Mokkadem, A. (1988).
\newblock Mixing properties of {ARMA} processes.
\newblock {\em Stochastic Processes and their Applications}, 29(2):309 -- 315.

\bibitem[Neumann, 2013]{neumann}
Neumann, M.~H. (2013).
\newblock A central limit theorem for triangular arrays of weakly dependent
  random variables, with applications in statistics.
\newblock {\em ESAIM: PS}, 17:120--134.

\bibitem[Romano and Tirlea, 2022]{tirlea}
Romano, J.~P. and Tirlea, M.~A. (2022).
\newblock Permutation testing for dependence in time series.
\newblock {\em Journal of Time Series Analysis}, 43:781 -- 807.

\bibitem[Romano and Tirlea, 2024]{tirlea-OLS}
Romano, J.~P. and Tirlea, M.~A. (2024).
\newblock Least squares-based permutation tests in time series.
\newblock {\em Technical Report, Department of Statistics, Stanford
  University}.

\bibitem[Volkonskii and Rozanov, 1961]{rozanov}
Volkonskii, V.~A. and Rozanov, Y.~A. (1961).
\newblock Some limit theorems for random functions. ii.
\newblock {\em Theory of Probability \& Its Applications}, 6(2):186--198.

\bibitem[Wang et~al., 2020]{power_mk_test}
Wang, F., Shao, W., Yu, H., Kan, G., He, X., Zhang, D., Ren, M., and Wang, G.
  (2020).
\newblock Re-evaluation of the power of the {Mann-Kendall} test for detecting
  monotonic trends in hydrometeorological time series.
\newblock {\em Frontiers in Earth Science}, 8.

\bibitem[Yue et~al., 2002]{yue2002}
Yue, S., Pilon, P., Phinney, B., and Cavadias, G. (2002).
\newblock The influence of autocorrelation on the ability to detect trend in
  hydrological series.
\newblock {\em Hydrological Processes}, 16(9):1807--1829.

\bibitem[Zhao and Woodroofe, 2012]{zhao_woodroofe}
Zhao, O. and Woodroofe, M. (2012).
\newblock Estimating a monotone trend.
\newblock {\em Statistica Sinica}, 22(1):359--378.

\end{thebibliography}

\end{document}